\documentclass[letterpaper, 10pt, conference]{IEEEtran}
\IEEEoverridecommandlockouts
\usepackage{float}
\usepackage{amsmath, amssymb, amsthm}
\usepackage{multirow}
\usepackage{empheq}
\usepackage{cite}
\usepackage{enumerate}
\usepackage[ruled,linesnumbered]{algorithm2e}
\usepackage[colorlinks=true,linkcolor=blue,citecolor=blue]{hyperref}
\usepackage{xcolor}
\makeatletter
\def\subsection{\@startsection{subsection}{2}{0em}%
  {4pt plus 1pt minus 1pt}%
  {2pt plus 1pt minus 1pt}%
  {\normalfont\normalsize\bfseries}}
\makeatother
\newtheorem{thm}{Theorem}
\newtheorem{lem}{Lemma}

\newtheorem{prop}{Proposition}
\newtheorem{ass}{Assumption}
\newtheorem{defin}{Definition}
\newtheorem{rem}{Remark}
\newenvironment{Theorem}{\begin{thm}}{\hfill$\square$\end{thm}}

\newenvironment{Assumption}{\begin{ass}}{\hfill$\bullet$\end{ass}}

\newenvironment{Remark}{\begin{rem}}{\hfill$\bullet$\end{rem}}

\newcommand{\R}{\mathbb{R}}
\newcommand{\dd}{\mathrm{d}}
\newcommand{\norm}[1]{\left\|#1\right\|}

\newcommand{\dmin}{d_{\min}}
\newcommand{\dmax}{d_{\max}}
\newcommand{\Rcal}{\mathcal{R}}
\newcommand{\Rhat}{\widehat{\mathcal{R}}}
\newcommand{\Rr}{\mathcal{R}_r}
\newcommand{\Ccal}{\mathcal{C}}
\newcommand{\dH}{d_{\mathrm{H}}}
\newcommand{\nh}{n_h}
\title{\vspace{-20pt}\bf Certified Reachable Sets for Nonlinear Reaction--Diffusion Systems}\vspace{-16pt}

\author{
    Mohamed Amine Ouchdiri$^{1}$, Mohamed Maghenem$^{2}$,
    Saad Benjelloun$^{3}$ and Adnane Saoud$^{1}$\\
    \thanks{$^{1}$College of Computing, University Mohammed VI Polytechnic, Benguerir, Morocco.
   $^{2}$Univ. Grenoble Alpes, CNRS, Grenoble-INP, GIPSA-lab, F-38000 Grenoble, France. $^{3}$De Vinci Higher Education, De Vinci Research Center, Paris, France.}
}

\begin{document}
\maketitle

\pagestyle{empty}
\begin{abstract}
Reachability analysis for dynamical systems seeks to compute a set containing all reachable states at a given time. Compared to ordinary differential equations (ODEs), the analysis of nonlinear reaction--diffusion PDEs with parametric uncertainties remains largely underexplored, due to the infinite-dimensional state space and the variety of solutions under different parameters. We address this through a three-step procedure: 1) Finite Element Methods (FEM)s to discretise the space and generate a finite-dimensional FEM-based model, 2) Proper Orthogonal Decomposition (POD) to build a Reduced-Order Model (ROM), and 3) set-based reachability-analysis methods applied to the ROM. We propose a framework that enables us to derive explicit upper bounds on the approximation errors introduced at each stage of the pipeline. In particular, we quantify the discrepancy between trajectories of the original PDE and those of the FEM-based discretization, as well as the error between the FEM-based model and the reduced-order model. Importantly, these bounds are shown to hold uniformly over the considered set of parameters. By combining these error estimates, we obtain an over-approximation of the reachable set of the original PDE. The approach is illustrated on the Allen--Cahn equation and a logistic growth PDE.
\end{abstract}
\vspace{-5.5pt}
\section{Introduction}\label{sec1}
\vspace{-3.5pt}
Reaction-diffusion equations model substances spreading while undergoing chemical or biological transformations, with applications in enzyme kinetics~\cite{murray2003}, population dynamics~\cite{pao1992}, and synthetic developmental biology~\cite{ouchdiri2025optimal}. When some parameters are uncertain, the reachability-analysis problem asks for the states the system can reach for all possible parameters.  

For ODEs, reachability analysis is a mature subject. Indeed, set-propagation methods, based on zonotopes~\cite{girard2005} and interval arithmetics~\cite{scott2013}, compute guaranteed 
over-approximations of the reachable set, and tools, such as CORA~\cite{althoff2021}, make these methods readily applicable.
For PDEs, however, reachability analysis is relatively under developed, in the sense of allowing general parameter dependency and nonlinear reactions. For instance, \cite{tran2018reachPDE}
constructs an over-approximation of the reachable set for the
one-dimensional linear parabolic equation, which is affine in the parameters, using FEM,
interpolation, and numerical error correction. Furthermore, for linear parabolic PDEs, under interval
uncertainty on the initial condition, the disturbance, and the
measurements,~\cite{kharkovskaia2016} uses FEM, tailored to a monotone interval observer, to produce upper and lower bounds on the PDE's solution.  This approach requires output measurements and a specific structure for the FEM-based model. 
For nonlinear reaction-diffusion PDEs, comparison theorems are used in \cite{pao1992} to construct intermediate systems whose solutions define an envelope that contains the PDE's solution. Applying this method to parametric
reachability would require knowing a parameter profile that generate 
a solution that upperbounds all the remaining ones for different admissible parameter profiles. However, this knowledge is hard to obtain in most applications.

This paper studies reachability analysis for nonlinear
reaction-diffusion PDEs, where the diffusion coefficient, the reaction term, and the initial condition involve 
unknown parameters. As a result, we construct an over-approximation of the reachable set, over all possible parameters values, using a three-step process. First, we use FEM to discretise the space and generate a finite-dimensional FEM-based model \cite{thomee2006}. Then, we leverage the FEM-based model using POD \cite{kunisch2001} to obtain a ROM. Finally, we apply well-established set-based reachability-analysis methods on the ROM to generate an over-approximation of the true reachable set for the original parametric PDE. Furthermore, we propose a set of assumptions on the reaction term, the initial condition, and the class of parametric uncertainties, under  which, we are able to explicitly estimate the margin between the calculated reachable set and the true one. In particular, thanks to \cite{thomee2006}, we are able to upperbound the norm of the error between trajectories of the original PDE and the FEM-based model. Furthermore, to upperbound the norm of the error between trajectories of the FEM-based model and the ROM, we build upon existing POD-based results \cite{kunisch2001} to derive new error upperbounds that are uniform over the considered set of parameters. Finally, we illustrate the effectiveness of the approach 
on the Allen--Cahn equation and a
logistic growth PDE.

The remainder of this paper is organized as follows. Section~\ref{sec2} introduces the PDE model, the required assumptions, and the three-step reduction approach. Section~\ref{sec3} states and proves the main result. Section~\ref{sec4} presents the numerical illustrations. Section~\ref{sec5} concludes the paper.

\textbf{Notation.} On $\R^n$, $\|\cdot\|_{\R^n}$ is the Euclidean norm and $\mathbb{B}(0,\eta):=\{z:\|z\|_{\R^n}\leq\eta\}$. On $\Omega\subset\R$, $(u,v):=\int_\Omega u\,v\,\dd x$, with norm $\|\cdot\|_{L^2}$; $\|\cdot\|_{L^\infty}$ is the essential-supremum norm; $H^k$ is the Sobolev space with seminorm $|v|_{H^k}:=\|\partial_x^k v\|_{L^2}$; $C^k$, $C^{k,1}$, $C^{k,l}$ denote the corresponding spaces of continuously differentiable, Lipschitz $k$-th derivative, and mixed regularity functions. For symmetric positive-definite $M$, $\|v\|_M:=(v^\top Mv)^{1/2}$; $I_r$ is the $r\times r$ identity; $d_H(A,B):=\sup_{b\in B}\inf_{a\in A}\|b-a\|$ for $A\subseteq B$ is the Hausdorff distance; $\mathrm{diam}(P):=\sup_{p,q\in P}\|p-q\|$.
\vspace{-4pt}
\section{Reduction of Reaction-Diffusion Systems} \label{sec2}
We consider the parametric reaction-diffusion PDE defined on the
spatial domain $\Omega=(0,L)$
over the time horizon $[0,T]$:
\begin{empheq}[left=\Sigma:\ \left\{,right=\right.]{align}
  &\partial_t u = d(p)\,\partial_{xx}u + f(u;p),
    &\text{in } &\Omega\times(0,T], \label{eq1}\\
  &\partial_x u = 0,
    &\text{on } &\partial\Omega\times(0,T], \label{eq2}\\
  &u = u_0(x;p),
    &\text{on } &\Omega\times\{0\}, \label{eq3}
\end{empheq}
where $u:\overline{\Omega}\times[0,T]\to\R$ is the state, $p \in P \in\R^{n_p}$, for some $n_p\geq 1$, is a constant but not precisely known parameter, $d : P \rightarrow \mathbb{R}_{\geq 0}$
is the diffusion coefficient, $f:\R\times P\to\R$ is the reaction
term, and $u_0(\cdot;p) \in C^2(\overline{\Omega})$ is the initial
condition such that $u_0(x;p)\geq 0$ for all $(x,p) \in\overline{\Omega} \times P$.

Equation~\eqref{eq2} imposes
homogeneous Neumann boundary conditions. Furthermore, solutions are understood in the classical sense, i.e., a solution
$u$ is in $C^{2,1}(\Omega\times(0,T])\cap C(\overline{\Omega}\times[0,T])$ and 
satisfies \eqref{eq1}--\eqref{eq3} pointwisely.

\begin{Assumption}\label{ass1}
There exist $\dmin, \dmax > 0$ such that 
\begin{align} \label{eqdass}
\dmin \leq d(p)\leq \dmax \qquad 
\forall p\in P.
\end{align}
The map $f$ is continuous and the gradient $\partial_u f$ is 
continuous on $[0,M] \times P$, for some $M > 0$, such that
\begin{equation}\label{eq4}
  f(0;p)\geq 0 \quad\text{and}\quad f(M;p)\leq 0
  \qquad\forall p\in P.
\end{equation}
Finally, we assume that $\sup_{p \in P} \norm{u_0(\cdot;p)}_{L^\infty}\leq M$.
\end{Assumption}

\begin{Remark}
Conditions~\eqref{eqdass}-\eqref{eq4} ensure positive invariance of the set $[0,M]$ for $\Sigma$; see~\cite[p.~199]{pao1992}. That is, given $p \in P$ and $ \norm{u_0(\cdot;p)}_{L^\infty}\leq M$, a unique global classical solution to $\Sigma$ exists and verifies 
$\norm{u(\cdot,\cdot;p)}_{L^\infty}\leq M$. Condition \eqref{eq4} is  satisfied for stable linear reactions. It is also satisfied for Allen--Cahn,
logistic, Fisher--KPP, Nagumo, and Hill-type reactions under
suitable parameter ranges.

\end{Remark}
Under Assumption~\ref{ass1}, 
for $L_f := \sup_{[0,M]\times P}|\partial_u f(u,p)|$
and $\mu: =\sup_{[0,M]\times P}\partial_u f(u;p)$, we conclude that, for all $a,b\in[0,M]$ and $p\in P$,
\begin{align}
  |f(a;p)-f(b;p)|&\leq L_f|a-b|, \label{eq5}\\
  (f(a;p)-f(b;p))(a-b)&\leq \mu|a-b|^2. \label{eq6}
\end{align}

Since $\Sigma$ is infinite-dimensional, set-based reachability tools cannot be applied directly. We proceed in three steps: FEM discretises $\Sigma$ into a finite-dimensional ODE system, POD reduces its dimension to a tractable ROM, and set-based reachability methods are applied to the ROM. The approximation errors at each step are combined to bound the gap to the true reachable set of $\Sigma$.
\subsection{Step~1: The FEM-Based Model}
FEMs~\cite{thomee2006} discretize the domain $\Omega$ using a uniform mesh with $\nh > 0$ nodes and mesh size $h := L/(\nh-1)$. The solution $u$ to $\Sigma$ is approximated by the combination
\begin{equation}\label{eq9}
  u_h(x,t;p) := \sum_{i=1}^{\nh} a_i(t;p)\,\varphi_i(x),
\end{equation}
where $\varphi_i$, $i \in \{1,...,\nh\}$, is the hat function that equals one at node $x_i$ and is equal to zero at all other nodes. 

The standard Galerkin projection yields the following ODE governing the vector $a(t;p) \in\R^{\nh}$ of coefficients of $u_h$.
\begin{equation}\label{eq10}
   M \dot{a} = -d(p) K a + F(a;p),
  ~~ a(0;p) = a_0(p),
\end{equation}
where $M,K\in\R^{\nh\times\nh}$ verify
$$M_{ij} := \int_\Omega\varphi_i(x)\,\varphi_j(x)\,\dd x, ~~ K_{ij} := \int_\Omega \partial_x \varphi_{i}(x)\, \partial_x \varphi_j(x)\,\dd x, $$ and $F_i(a;p) := \int_\Omega f(u_h(x,t;p);p)\,\varphi_i(x)\,\dd x$. 

The following result provides an upper bound on the mismatch between the true solution $u$ to $\Sigma$ and its 
FEM-based approximation.

\begin{lem}[{\cite[Ch.~14, Th.~14.1]{thomee2006}}]
\label{prop1}
Consider $\Sigma$ such that Assumption~\ref{ass1} holds. 
Let $u_h(x,t;p)$ be the FEM approximation defined by~\eqref{eq9}, where the coefficient vector $a(t;p)$ is the solution of~\eqref{eq10} with initial condition satisfying
\begin{equation}
\label{InitCd}
\begin{aligned}
a_0(p) & := M^{-1} b(p) \qquad \forall p \in P, 
\\
b_i(p) & := \int_\Omega u_0(x;p)\,\varphi_i(x)\,\dd x \qquad  \forall i \in \{1,...,n_h\}. 
\end{aligned}
\end{equation}
Then, for all $(t,p) \in [0,T] \times P$,
\begin{equation}\label{eq17}
  \|u(\cdot,t;p)-u_h(\cdot,t;p)\|_{L^2(\Omega)}
  \leq \varepsilon_h, 
\end{equation}
where $\varepsilon_h>0$ depends on $T$, $\dmin$, $\dmax$, $L_f$, $h$, and $M$. 
\end{lem}

For the FEM-based model in~\eqref{eq10} to be accurate, a large $\nh$ is usually required, which makes the reachable-set computation over the set of parameters $P$ intractable. Hence, the following subsection provides a 
POD-based approach to reduce the dimension of the obtained FEM-based model. 

\subsection{Step~2: ROM via POD}
POD-based approch in\cite{kunisch2001}
suggest solving~\eqref{eq10} for parameter samples
$\{p_1,\ldots,p_{n_p}\}\subset P$ at time instances
$\{t_1,\ldots,t_{n_t}\}\subset[0,T]$, with a constant time step $\tau := T/(n_t - 1)$.  

The collected
snapshots $a(t_k;p_j)\in\R^{\nh}$ and 
the first-order difference quotients
\begin{equation}\label{eq20}
  \overline\partial a(t_k;p_j) :=
  \frac{a(t_{k+1};p_j)-a(t_k;p_j)}{t_{k+1}-t_k} \in\R^{\nh}
\end{equation}
are gathered in the matrix 
$S\in\R^{\nh n_t \times 2 n_p}$ so that
\begin{equation}
\label{eqSSS}
\begin{aligned}
S  & := [S_1~...~S_{2 n_p}] \quad \text{and}\quad \forall i \in \{1,...,n_p \}:
\\
S_i & := [a(t_1;p_i)^\top ~ ... ~ a(t_{n_t};p_i)^\top ]^\top,
\\ 
S_{n_p + i} & := [1_{n_h}^\top ~ \bar{\partial }a(t_1;p_i)^\top ~ ... ~ \bar{\partial } a(t_{n_t-1};p_i)^\top ]^\top.
\end{aligned}
\end{equation}

The singular-value
decomposition of $S$ produces $\sigma_1\geq\cdots\geq\sigma_{2 n_p}\geq 0$.
Furthermore, a POD basis $V\in\R^{\nh \times r}$ is formed by the first $r$
left singular vectors of $S$, orthonormalised with respect to the mass
matrix $M$ so that
\begin{equation}\label{eq11}
  V^\top M V = I_r, \quad\text{hence}\quad
  \|Vc\|_{M} = \|c\|_{\R^r}\;\;\forall\,c\in\R^r.
\end{equation}

Substituting $a=Vc$ into~\eqref{eq10} and projecting onto the
column space of $V$ produces the ROM
\begin{equation}\label{eq13}
  \frac{d c}{dt}=g(c;p),
  \qquad c(0;p)=V^\top a_0(p),
\end{equation}
with state $c(t;p)\in\R^r$ and reduced right-hand side
\begin{equation}\label{eq14}
  g(c;p):=V^\top M^{-1}\bigl(-d(p)\,K\,Vc+F(Vc;p)\bigr).
\end{equation}
The corresponding output is 
\begin{equation}\label{eq15}
  u_r(x,t;p):=\sum_{i=1}^{\nh}(Vc(t;p))_i\,\varphi_i(x).
\end{equation}

At this point, we recall the following result showing optimality of the POD basis.
\begin{lem}[{\cite[Th.~2.1]{kunisch2001}}]
\label{prop2}
 Let $a(t_k;p_j) \subset \R^{\nh}$, $k,j \in \{1,...,n_t\} \times \{1,...,n_p \}$, be the snapshots of the FEM-based model in~\eqref{eq10} and let the corresponding matrix $S$ in \eqref{eqSSS} and the resulting POD basis $V \in \R^{\nh \times r}$. Then, for any matrix $W \in \R^{\nh \times r}$ with $W^\top M W = I_r$,
\begin{equation}
\label{eq19}
\begin{aligned}
 \sum_{i=r+1}^{n_s}\sigma_i^2 & = \sum_{k,j}
  \|a(t_k;p_j) - VV^\top a(t_k;p_j)\|_M^2
  \\ &
  \leq
  \sum_{k,j}
  \|a(t_k;p_j) - WW^\top a(t_k;p_j)\|_M^2.
\end{aligned}
\end{equation}
\end{lem}
Lemma~\ref{prop2} shows that $V$ represents the snapshot data better then any other basis choice. 
Furthermore, since the matrix $S$ includes the derivatives in \eqref{eq20}, we are able to show that  
$$ \sup_{p\in P} \left\{ \sum_{k}
  \|a(t_k;p) - VV^\top a(t_k;p)\|_M^2 \right\} < +\infty. $$ 
The latter will play a key role in the proof of   Proposition~\ref{prop3} to establish the existence of $\varepsilon_r > 0$ such that 
\begin{equation} \label{eqADDEd}
  \sup_{p\in P}\sup_{t\in[0,T]}
  \|u_h(\cdot,t;p)-u_r(\cdot,t;p)\|_{L^2}\leq\varepsilon_r. 
\end{equation}
\vspace*{-8pt}
\subsection{Step~3: Reachability Analysis on the ROM}
The reachable set of $\Sigma$ at time $t\in[0,T]$ is
\begin{equation}\label{reach_set}
\Rcal(t):=\{u(\cdot,t;p):p\in P\}.  
\end{equation}

The ROM in~\eqref{eq13}
still depends on $p\in P$ and generates the reduced reachable
set 
\begin{align} \label{eqRr} 
\Rr(t):=\{c(t;p):p\in P\}\subset\R^r.
\end{align}
At this point, we can use existing set-based
reachability-analysis methods, based on  \cite{girard2005}  or \cite{althoff2010}, to compute a set
$\Ccal(t)\subset\R^r$ guaranteed to contain $\Rr(t)$. We can also, under further conditions, guarantee the existence of $\eta>0$ such that $d_H(\Rr(t),\Ccal(t)) \leq \eta$.

\begin{Remark}\label{rem1}
 Note that the $\eta$ mismatch between the true reachable set $\mathcal{R}_r(t)$ and its approximation $\mathcal{C}(t)$ 
is usually achieved under specific structural properties of $g$ and the parameter set $P$. For example, when $g$ is monotone and $P$ is an
interval~\cite{meyer2018}, (b) holds with $\eta=0$. When
$g$ is linear with $P$ represented by support
functions~\cite{le2010reachability}  or when $g$ is nonlinear
and $P$ is compact and convex~\cite{rungger2018accurate} any prescribed $\eta>0$ can be achieved. 
When $g$ is a general Lipschitz function, CORA~\cite{althoff2010}
achieves (b) with $\eta$ dependent on a discretization step size and
$\mathrm{diam}(P)$.
\end{Remark}

\section{Main Result} \label{sec3}

In this section, we seek a computable set 
$\Rhat(t)$ that contains $\Rcal(t)$ such that
$\dH(\Rhat(t),\Rcal(t))\leq\varphi(\varepsilon_h,\varepsilon_r,\eta)$, for a computable $\varphi$.
To this end, we propose an original result that computes $\varepsilon_r$, for which \eqref{eqADDEd} holds, using the following assumption.   
\vspace*{-5pt}
\begin{Assumption}\label{ass2}
There exists a positive integer $Q$ such that, for every $l \in \{1,2,...,Q \}$, there exist
$f_l \in C^{1,1}([0,M])$ and 
$u^l_0 \in C^3(\overline{\Omega})$ with $\partial_x u^l_0|_{\partial\Omega}=0$,   there exist Lipschitz functions
$\theta_l, \xi_l : P \to \R$ with Lipschitz  constants
$L_\theta, L_\xi$, respectively, 
such that, for all
$p \in P$,
\begin{equation}\label{eq8}
\begin{aligned}
f(u;p) = \sum^{Q}_{l=1} \theta_l(p)\,f_l(u), \quad 
u_0^l(\cdot;p) = \xi_l(p)\,u^l_0(\cdot).
\end{aligned}
\end{equation}
Finally, $d$ is Lipschitz with Lipschitz constant $L_d$.
\end{Assumption}
\vspace{-4pt}
Next, we introduce the Ritz projector $R_r: X_h \to X_r$, where $X_h := \mathrm{span}\{\varphi_1, \ldots, \varphi_{n_h}\}$ is the FEM space generated by the hat functions in~\eqref{eq9}, and $X_r$ is the $r$-dimensional subspace of $X_h$ generated by the columns of $V$ in~\eqref{eq11}, i.e.,
\begin{equation*}
X_r := \Big\{\sum_{i=1}^{n_h}(Vc)_i\, \varphi_i \;:\; c \in \mathbb{R}^r\Big\}.
\end{equation*}
The projector $R_r$ is defined by the orthogonality condition
\begin{equation}\label{eq16}
  (\partial_x(w - R_r w), \partial_x v_r) = 0 \quad \forall v_r \in X_r.
\end{equation}
% As a The following proposition provides an upperbound on the error between $u_h$ in~\eqref{eq9} and $u_r$ in~\eqref{eq15}: first for each $p \in P$ under Assumption~\ref{ass1}, then in the sense of~\eqref{eqADDEd} under Assumptions~\ref{ass1} and~\ref{ass2}.
\vspace{-20pt}
\begin{prop} \label{prop3}
Let the FEM-based model in~\eqref{eq10} generate the output $u_h$ given by \eqref{eq9} and let the corresponding ROM in~\eqref{eq13} generate the output $u_r$ given by~\eqref{eq15}. Under  Assumption \ref{ass1} and for initial conditions verifying 
\begin{align} \label{initCd1}
u_r(\cdot,0;p) = R_r u_h(\cdot,0;p), \qquad  p \in P,
\end{align} 
where $R_r$ is the Ritz projector verifying~\eqref{eq16}, 
we have 
\begin{equation}\label{eq23}
  \|u_h(\cdot,t;p) - u_r(\cdot,t;p)\|_{L^2} \leq \varepsilon_r(p) \quad \forall t \in [0,T],
\end{equation}
where
\begin{align*}
  \varepsilon_r(p)^2 := C_0 E(p), ~~ 
  C_0  := 2 + 2 e^{(2\mu + 1)T}(1 + L_f^2 T),
\end{align*}
\begin{multline}\label{eq22}
  E(p) := \sup_{t\in[0,T]}\|u_h(\cdot,t;p) - R_r u_h(\cdot,t;p)\|_{L^2}^2 \\
  + \int_0^T \|\partial_t(u_h - R_r u_h)(\cdot,t;p)\|_{L^2}^2\, \dd t.
\end{multline}
When Assumption \ref{ass2} additionally holds, then there exist  $C_1, C_2 > 0$ (depending on $T$, $\dmin$, $\dmax$, $M$, $L_f$, $\mu$, $L$, $L_d$, $L_\theta$, and $L_\xi$) such that, when $\tau$ is sufficiently small, we conclude, for 
$ \rho_P := \sup_{p\in P}\min_{1\leq j\leq n_p}\|p-p_j\|_{\R^{n_p}} $, that 
\begin{equation}\label{eq25}
  \sup_{p \in P} E(p) \leq C_1 \sum_{i=r+1}^{n_s}\sigma_i^2 + C_2 \rho_P.
\end{equation} 
Hence, \eqref{eqADDEd} holds for
  $\varepsilon_r^2 :=  C_0\Bigl[C_1 \sum_{i=r+1}^{n_s}\sigma_i^2 + C_2 \rho_P\Bigr]$.
\end{prop}

In the following theorem, given $t \in [0,T]$, we build a set $\Rhat(t)$ approximating $\mathcal{R}(t)$ in \eqref{reach_set} by mapping the set $\Ccal(t)$, approximating the ROM reachable set $\mathcal{R}_r(t)$ in \eqref{eqRr}, back to $X_h := \mathrm{span}\{\varphi_1, \ldots, \varphi_{n_h}\}$ using $V$, and  
 enlarging by $\varepsilon_h+\varepsilon_r$ for the  constants $\varepsilon_h$ and $\varepsilon_r > 0$ introduced in  Lemma~\ref{prop1} and Proposition~\ref{prop3}, respectively. That is, we let $\Phi(\hat{c}) := \sum_{i=1}^{\nh}(V\hat{c})_i\,\varphi_i$ and 
\begin{multline}\label{eq37_1}
  \Rhat(t):=\bigl\{v\in L^2(\Omega):
  \exists\hat{c}\in\Ccal(t):\\
  \|v-\Phi(\hat{c})\|_{L^2}\leq
  \varepsilon_h+\varepsilon_r\bigr\}.
\end{multline}
\vspace{-20pt}
\begin{Theorem} \label{thm1}
Let $\Sigma$ verify Assumptions~\ref{ass1} and~\ref{ass2} and generate classical solutions $u(\cdot,\cdot;p)$ for all $p \in P$. Let the corresponding FEM-based model in~\eqref{eq10} generate the outputs $u_h(\cdot,\cdot;p)$ in~\eqref{eq9} for all $p \in P$ and for initial conditions verifying \eqref{InitCd}.  
Let the ROM in~\eqref{eq13} generate  the outputs $u_r(\cdot,\cdot;p)$ in~\eqref{eq15} for all $p \in P$ and for initial conditions verifying  \eqref{initCd1}. 
Let $\varepsilon_h, \varepsilon_r > 0$ such that~\eqref{eq17} and \eqref{eqADDEd} hold, respectively. 
Then, given $t \in [0,T]$ and an approximation $\Ccal(t)$ of the ROM reachable set $\mathcal{R}_r(t)$ in \eqref{eqRr}, the following hold:
\begin{enumerate}[(i)]
  \item $\Rr(t) \subset \Ccal(t)$ $\Rightarrow$  $\Rcal(t)\subseteq\Rhat(t)$;
  \item  $d_H(\Rr(t),\Ccal(t)) \leq \eta$ $\Rightarrow$  $\dH\!\bigl(\Rhat(t),\Rcal(t)\bigr) \leq 2\varepsilon_h+2\varepsilon_r+\eta$.
\end{enumerate}
\end{Theorem}
\vspace{-18pt}
\begin{proof}
Fix $t \in [0,T]$ and $p \in P$, and write $u_r(\cdot, t; p) = \Phi(c(t; p))$ with $c(\cdot; p)$ the solution of~\eqref{eq13}. The triangle inequality gives
\begin{equation}
\label{eq38}
\begin{aligned} 
  \|u(\cdot,t;p) - \Phi(c(t;p))\|_{L^2}   & \leq 
  \\
 \|u(\cdot,t;p) - u_h(\cdot,t;p)\|_{L^2} + &  \|u_h(\cdot,t;p) - u_r(\cdot,t;p)\|_{L^2}.
\end{aligned}
\end{equation}
According to Lemma~\ref{prop1}, the first term on the right-hand side  of~\eqref{eq38} is bounded by $\varepsilon_h$, and by Proposition~\ref{prop3}, the second is bounded by $\varepsilon_r$, therefore
\begin{equation}\label{eq39}
  \|u(\cdot,t;p) - \Phi(c(t;p))\|_{L^2} \leq \varepsilon_h + \varepsilon_r.
\end{equation}

\begin{itemize}
\item  When $\Rr(t) \subset \Ccal(t)$, we conclude that $\Rr(t) \subseteq \Ccal(t)$ and; thus,  $c(t; p) \in \Ccal(t)$.
Hence, by~\eqref{eq39}, we conclude that  $u(\cdot, t; p) \in \Rhat(t)$ in~\eqref{eq37_1}.
\item When $d_H(\Rr(t),\Ccal(t)) \leq \eta$, we prove that 
\begin{multline*}
  \dH(\Rhat(t), \Rcal(t)) = \sup_{v \in \Rhat(t)} \inf_{p \in P} \|v - u(\cdot,t;p)\|_{L^2} \\
  \leq 2\varepsilon_h + 2\varepsilon_r + \eta.
\end{multline*}
For this, we pick $v \in \Rhat(t)$. By~\eqref{eq37_1}, there exists $\hat c \in \Ccal(t)$ such that
\begin{equation}\label{eq40}
  \|v - \Phi(\hat c)\|_{L^2} \leq \varepsilon_h + \varepsilon_r.
\end{equation}
Since $\hat c \in \Ccal(t)$ and $\dH(\Ccal(t), \Rr(t)) \leq \eta$, there must exist $c^* \in \Rr(t)$ such that $\|\hat c - c^*\|_{\R^r} \leq \eta$. By definition of $\Rr(t)$, we conclude the existence  of $p^* \in P$ with $c^* = c(t; p^*)$. By the definition of $\Phi$ and since  $\|\Phi(d)\|_{L^2} = \|d\|_{\R^r}$ for all $d \in \R^r$, we conclude that  
\begin{equation}\label{eq41}
  \|\Phi(\hat c) - \Phi(c^*)\|_{L^2} = \|\hat c - c^*\|_{\R^r} \leq \eta.
\end{equation}
Combining~\eqref{eq39}-\eqref{eq41}, we obtain 
\begin{multline*}
  \|v - u(\cdot,t;p^*)\|_{L^2} \leq \|v - \Phi(\hat c)\|_{L^2} \\
  + \|\Phi(\hat c) - \Phi(c^*)\|_{L^2} + \|\Phi(c^*) - u(\cdot,t;p^*)\|_{L^2} \\
  \leq (\varepsilon_h + \varepsilon_r) + \eta + (\varepsilon_h + \varepsilon_r) = 2\varepsilon_h + 2\varepsilon_r + \eta.
\end{multline*}The proof is completed
since $v \in \Rhat(t)$ is arbitrary.  
\end{itemize}

\end{proof}

\vspace{-8pt}\subsection{Proof of Proposition~\ref{prop3}}
\subsubsection{The FEM-ROM error bound at each $p \in P$}~\\[2pt]

Given $p \in P$,   testing~\eqref{eq10} with $v_h \in X_h$ and integrating by parts,  we obtain,  for all $v_h \in X_h$,
\begin{equation}\label{eqWF_FEM}
  (\partial_t u_h, v_h) + d(p) (\partial_x u_h, \partial_x v_h) = (f(u_h; p), v_h).
\end{equation}
The ROM output $u_r$ in~\eqref{eq15}, being a linear combination of basis functions in $X_r \subseteq X_h$, 
satisfies, for all $v_r \in X_r$,
\begin{equation}\label{eqWF_ROM}
  (\partial_t u_r, v_r) + d(p)(\partial_x u_r, \partial_x v_r) = (f(u_r; p), v_r).
\end{equation}
With these two weak forms at hand,  we let 
\begin{equation}\label{eqSplit}
 \psi := u_h - R_r u_h \in X_h,  ~~ \phi := R_r u_h - u_r \in X_r.
\end{equation} 
%The projection error $\psi$ is controlled by the POD approximation quality and feeds into the bound through the functional $E(p)$ in~\eqref{eq22}; the second component $\phi$ is the residual that we estimate by testing the difference of~\eqref{eqWF_FEM} and~\eqref{eqWF_ROM} against $\phi$ itself.
Subtracting~\eqref{eqWF_ROM} and \eqref{eqWF_FEM},  while testing with $\phi$,  we obtain
\begin{multline}\label{eq28}
  (\partial_t (u_h - u_r), \phi) + d(p)(\partial_x(u_h - u_r), \partial_x \phi) \\
  = (f(u_h; p) - f(u_r; p), \phi).
\end{multline}
Using Ritz orthogonality~\eqref{eq16},  and~\eqref{eqSplit},  
we conclude that $ (\partial_x(u_h - u_r), \partial_x \phi) =   \|\partial_x \phi\|_{L^2}^2 \geq 0$.   Applying the same approach to $(\partial_t (u_h - u_r), \phi)$,  we obtain  
\begin{equation}\label{eq29}
  \tfrac{1}{2} \tfrac{d}{dt} \|\phi\|_{L^2}^2 \leq -(\partial_t \psi, \phi) + (f(u_h; p) - f(u_r; p), \phi).
\end{equation}
The Lipschitz bound~\eqref{eq5} combined with Young's inequality, the one-sided condition~\eqref{eq6}, and Young's inequality on the time-derivative term give the bounds on the three terms of~\eqref{eq29}, which collect into
\begin{equation}\label{eq30}
  \tfrac{d}{dt}\|\phi\|_{L^2}^2 \leq (2\mu + 1)\|\phi\|_{L^2}^2 + \|\partial_t \psi\|_{L^2}^2 + L_f^2 \|\psi\|_{L^2}^2.
\end{equation}
Now, having $\phi(\cdot, 0) = R_r u_h(\cdot, 0; p) - u_r(\cdot, 0; p) = 0$, we apply Gronwall's lemma to obtain 
\begin{equation*}
  \sup_{t \in [0,T]}\|\phi(\cdot, t)\|_{L^2}^2 \leq e^{(2\mu+1)T}\bigl(1 + L_f^2 T\bigr)\, E(p).
\end{equation*}
Hence,  \eqref{eq23} follows since $\|u_h - u_r\|_{L^2} \leq \|\psi\|_{L^2} + \|\phi\|_{L^2}$.

\subsubsection{The uniform FEM-ROM error bound}
~\\[2pt]
We first establish Lipschitzness of the map  $p \mapsto u_h(t,x;p)$.   To do so,  we pick  $p, q \in P$ and consider the error
\begin{equation}\label{eqW}
  w(\cdot, \cdot;p;q) := u_h(\cdot, \cdot; p) - u_h(\cdot, \cdot; q).
\end{equation}

Writing~\eqref{eqWF_FEM} at $p$ and at $q$, subtracting, and testing with $w \in X_h$, leads to
\begin{equation}\label{eqDiffEq}
  \tfrac{1}{2}\tfrac{d}{dt}\|w\|_{L^2}^2 + d(p)\|\partial_x w\|_{L^2}^2 = -\mathcal{A}(t) + \mathcal{B}(t),
\end{equation}
\begin{equation}\label{eqA}
  \mathcal{A}(t) := \bigl(d(p) - d(q)\bigr)\bigl(\partial_x u_h(\cdot, t; q),\, \partial_x w\bigr),
\end{equation}
\begin{equation}\label{eqB}
  \mathcal{B}(t) := \bigl(f(u_h(\cdot, t; p); p) - f(u_h(\cdot, t; q); q),\, w\bigr).
\end{equation}

Using Assumption~\ref{ass2}, we obtain $|d(p) - d(q)| \leq L_d \|p - q\|$. Applying Young's inequality, we find
\begin{multline}\label{eqAbound}
  |\mathcal{A}(t)| \leq \tfrac{L_d^2}{2d_{\min}}\|p - q\|^2 \|\partial_x u_h(\cdot, t; q)\|_{L^2}^2 
  + \tfrac{d_{\min}}{2}\|\partial_x w\|_{L^2}^2.
\end{multline}

Next, we note that
\begin{multline}\label{eqSplitB}
  f(u_h(\cdot, t; p); p) - f(u_h(\cdot, t; q); q) \\
  = \bigl[f(u_h(\cdot, t; p); p) - f(u_h(\cdot, t; q); p)\bigr] \\
  + \bigl[f(u_h(\cdot, t; q); p) - f(u_h(\cdot, t; q); q)\bigr].
\end{multline}
Using Lipschitzness of $f$ and Assumption~\ref{ass2}, we obtain
\begin{equation}\label{eqBbound}
  |\mathcal{B}(t)| \leq L_f\, \|w\|_{L^2}^2 + \tfrac{K_f^2\, L}{2}\, \|p - q\|^2 + \tfrac{1}{2}\, \|w\|_{L^2}^2,
\end{equation}
\begin{equation}\label{eqKf}
  K_f := Q\, L_\theta\, \max_{l \in \{1, \ldots, Q\}} \max_{u \in [0, M]} |f_l(u)|.
\end{equation}
Cauchy--Schwarz and Young's inequalities were used to obtain~\eqref{eqBbound}.

Substituting~\eqref{eqAbound} and~\eqref{eqBbound} into~\eqref{eqDiffEq}, we obtain
\begin{equation}\label{eqGronwallReady}
  \tfrac{d}{dt}\|w\|_{L^2}^2 \leq (2 L_f + 1)\, \|w\|_{L^2}^2 + \kappa(t)\, \|p - q\|^2,
\end{equation}
\begin{equation}\label{eqKappa}
  \kappa(t) := K_f^2\, L + \tfrac{L_d^2}{d_{\min}}\, \|\partial_x u_h(\cdot, t; q)\|_{L^2}^2.
\end{equation}

Next, we use Assumption~\ref{ass2} to conclude that $\|w(\cdot, 0;p;q)\|_{L^2} \leq K_0 \|p - q\|$, with
\begin{equation}\label{eqK0}
  K_0 := Q\, L_\xi\, \max_{l \in \{1, \ldots, Q\}} \|u_0^l\|_{H^1(\Omega)}.
\end{equation}
Furthermore, we test~\eqref{eqWF_FEM} with $v_h = u_h$ to obtain
\begin{equation}\label{eqadded2}
  \tfrac{1}{2}\tfrac{d}{dt}\|u_h\|_{L^2}^2 + d(p)\|\partial_x u_h\|_{L^2}^2 = (f(u_h;p), u_h).
\end{equation}
The discrete maximum principle on the uniform mesh~\cite[Lemma~6.1, p.~96]{thomee2006} gives $\|u_h\|_{L^\infty} \leq M$, hence $|(f(u_h;p), u_h)| \leq M_1\, M\, L$, where
\begin{equation}\label{eqM1}
  M_1 := \sup_{(u,p) \in [0,M] \times P} |f(u;p)|.
\end{equation}
Having $d(p) \geq d_{\min}$ and integrating~\eqref{eqadded2} over $[0,T]$, we find
\begin{equation}\label{eqDxUh}
  \int_0^T \|\partial_x u_h(\cdot, t; q)\|_{L^2}^2\, \dd t \leq \tfrac{L\, (M^2 + 2\, T\, M_1\, M)}{2\, d_{\min}}.
\end{equation}

As a result, Gronwall's lemma applied to~\eqref{eqGronwallReady}, combined with~\eqref{eqDxUh}, yields
\begin{equation}\label{eq33}
  \sup_{t \in [0, T]} \|w(\cdot, t;p;q)\|_{L^2}^2 \leq G\, \|p - q\|^2,
\end{equation}
where $G$ is the explicit constant
\begin{multline}\label{eqG}
  G := e^{(2 L_f + 1) T} \Big[K_0^2 + L K_f^2 T 
  + \tfrac{L_d^2 L(M^2 + 2 T M_1 M)}{2 d_{\min}^2}\Big].
\end{multline}

As a second step, we establish Lipschitzness of $p \mapsto \partial_t u_h(t,x;p)$. Indeed, under Assumption~\ref{ass2}, the right-hand side of the FEM coefficient ODE
\begin{equation}\label{eqGp}
  \dot a = M^{-1}\bigl(-d(p)\, K\, a + F(a;p)\bigr)
\end{equation}
is $C^{1,1}$ on $\R^{n_h}$, so by~\cite[Th.~1.1, p.~8]{hartman1982} the solution $a(\cdot;p)$ is $C^2$ on $[0,T]$. Hence the weak form~\eqref{eqWF_FEM} can be differentiated in time, giving for every $v_h \in X_h$,
\begin{multline}\label{eqWF_dt}
  (\partial_{tt} u_h, v_h) + d(p)\, (\partial_x \partial_t u_h, \partial_x v_h) \\
  = (\partial_u f(u_h;p)\, \partial_t u_h,\, v_h).
\end{multline}

The Lipschitz analysis of $\partial_t u_h$ requires uniform-in-$p$ bounds on $\partial_t u_h$, $\partial_x \partial_t u_h$, and $\partial_{tt} u_h$. Evaluating~\eqref{eqWF_FEM} at $t=0$, using the Neumann condition $\partial_x u_0^l|_{\partial \Omega}=0$ and the Ritz property of $\Pi_h$, we obtain
\begin{multline}\label{eqDtUh0}
  \partial_t u_h(\cdot, 0; p)
  = P_{L^2}^{X_h}\bigl[d(p) \partial_{xx} u_0(\cdot;p) + f(\Pi_h u_0(\cdot;p);p)\bigr],
\end{multline}
where $P_{L^2}^{X_h}: L^2(\Omega) \to X_h$ is the $L^2$-orthogonal projector. Define the initial-profile bounds, finite by Assumption~\ref{ass2}, as
\begin{equation}\label{eqDk}
  D_k := Q\, \sup_{p \in P} |\xi_l(p)|\, \max_{l \in \{1,\ldots,Q\}} \|\partial_x^k u_0^l\|_{L^2}, \;\; k \in \{1,2,3\}.
\end{equation}

The $L^2$-stability $\|P_{L^2}^{X_h} v\|_{L^2} \leq \|v\|_{L^2}$, combined with~\eqref{eqM1} and~\eqref{eqDk}, produces
\begin{equation}\label{eqD0}
  \|\partial_t u_h(\cdot, 0; p)\|_{L^2} \leq D_0,
\end{equation}
\begin{equation}\label{eqD0def}
  D_0 := d_{\max}\, D_2 + M_1\, \sqrt{L}.
\end{equation}
Applying $\partial_x$ to~\eqref{eqDtUh0} and using the $H^1$-stability of $P_{L^2}^{X_h}$~\cite[Theorem~4]{crouzeix1987},
\begin{equation}\label{eqCP}
  \|P_{L^2}^{X_h} v\|_{H^1(\Omega)} \leq C_P\, \|v\|_{H^1(\Omega)} \qquad \forall v \in H^1(\Omega),
\end{equation}
together with $\|\partial_x \Pi_h u_0(\cdot;p)\|_{L^2} \leq D_1$ from the definition of $\Pi_h$, we obtain
\begin{equation}\label{eqD4}
  \|\partial_x \partial_t u_h(\cdot, 0; p)\|_{L^2} \leq D_4,
\end{equation}
\begin{equation}\label{eqD4def}
  D_4 := C_P\, \sqrt{D_0^2 + (d_{\max}\, D_3 + L_f\, D_1)^2}.
\end{equation}

Testing~\eqref{eqWF_FEM} with $v_h = \partial_t u_h$, integrating over $[0,T]$, dropping the non-negative terminal $\tfrac{d(p)}{2}\|\partial_x u_h(\cdot, T; p)\|_{L^2}^2$, and bounding the initial $\tfrac{d(p)}{2}\|\partial_x u_h(\cdot, 0; p)\|_{L^2}^2$ above by $\tfrac{d_{\max}}{2} D_1^2$ via~\eqref{eqDk}, we obtain
\begin{equation}\label{eqM2}
  \int_0^T \|\partial_t u_h(\cdot, t; p)\|_{L^2}^2\, \dd t \leq M_2,
\end{equation}
\begin{equation}\label{eqM2def}
  M_2 := T\, M_1^2\, L + d_{\max}\, D_1^2.
\end{equation}

Testing~\eqref{eqWF_dt} with $v_h = \partial_t u_h$ produces the identity
\begin{multline}\label{eqDtIdent}
  \tfrac{1}{2}\tfrac{d}{dt}\|\partial_t u_h\|_{L^2}^2 + d(p)\, \|\partial_x \partial_t u_h\|_{L^2}^2 \\
  = (\partial_u f(u_h;p)\, \partial_t u_h,\, \partial_t u_h).
\end{multline}
Integrating~\eqref{eqDtIdent} over $[0,T]$, dropping the non-negative terminal term, bounding the initial $\tfrac{1}{2}\|\partial_t u_h(\cdot,0;p)\|_{L^2}^2$ by $\tfrac{1}{2} D_0^2$ via~\eqref{eqD0}, and using $|\partial_u f| \leq L_f$ together with~\eqref{eqM2}, we obtain
\begin{equation}\label{eqE1}
  \int_0^T \|\partial_x \partial_t u_h(\cdot, t; p)\|_{L^2}^2\, \dd t \leq E_1,
\end{equation}
\begin{equation}\label{eqE1def}
  E_1 := \tfrac{D_0^2 + 2\, L_f\, M_2}{2\, d_{\min}}.
\end{equation}

Testing~\eqref{eqWF_dt} with $v_h = \partial_{tt} u_h$ produces the identity
\begin{multline}\label{eqDttIdent}
  \|\partial_{tt} u_h\|_{L^2}^2 + \tfrac{d(p)}{2}\tfrac{d}{dt}\|\partial_x \partial_t u_h\|_{L^2}^2 \\
  = (\partial_u f(u_h;p) \partial_t u_h,\partial_{tt} u_h).
\end{multline}
Young's inequality on the right-hand side, integration over $[0,T]$, dropping the non-negative terminal term, bounding the initial $\tfrac{d(p)}{2}\|\partial_x \partial_t u_h(\cdot,0;p)\|_{L^2}^2$ by $\tfrac{d_{\max}}{2} D_4^2$ via~\eqref{eqD4}, together with~\eqref{eqM2}, give
\begin{equation}\label{eqTheta}
  \int_0^T \|\partial_{tt} u_h(\cdot, t; p)\|_{L^2}^2\, \dd t \leq \Theta,
\end{equation}
\begin{equation}\label{eqThetaDef}
  \Theta := d_{\max}\, D_4^2 + L_f^2\, M_2.
\end{equation}

Writing~\eqref{eqWF_dt} at $p$ and at $q$, subtracting, and testing with $\partial_t w \in X_h$ produce an identity similar to~\eqref{eqDiffEq}:
\begin{multline}\label{eqDiffEqW}
  \tfrac{1}{2}\tfrac{d}{dt}\|\partial_t w\|_{L^2}^2 + d(p)\, \|\partial_x \partial_t w\|_{L^2}^2 
  = -\widetilde{\mathcal{A}}(t) + \widetilde{\mathcal{B}}(t),
\end{multline}
\begin{equation}\label{eqAtilde}
  \widetilde{\mathcal{A}}(t) := (d(p) - d(q))(\partial_x \partial_t u_h(\cdot, t; q),\, \partial_x \partial_t w),
\end{equation}
\begin{multline}\label{eqBtilde}
  \widetilde{\mathcal{B}}(t) := \bigl(\partial_u f(u_h(p); p)\, \partial_t u_h(\cdot, t; p) \\
  - \partial_u f(u_h(q); q)\, \partial_t u_h(\cdot, t; q),\, \partial_t w\bigr).
\end{multline}

The bound on $\widetilde{\mathcal{A}}$ replicates~\eqref{eqAbound} with $\partial_x u_h$ replaced by $\partial_x \partial_t u_h$:
\begin{multline}\label{eqAtildeBound}
  |\widetilde{\mathcal{A}}(t)| \leq \tfrac{L_d^2}{2\, d_{\min}}\, \|p - q\|^2\, \|\partial_x \partial_t u_h(\cdot, t; q)\|_{L^2}^2 \\
  + \tfrac{d_{\min}}{2}\,\|\partial_x \partial_t w\|_{L^2}^2.
\end{multline}
For $\widetilde{\mathcal{B}}$, the $C^{1,1}$ regularity of $f$ in Assumption~\ref{ass2} gives the Lipschitz constant of $\partial_u f$ in $u$,
\begin{equation}\label{eqLfPrime}
  L_f' := \sup_{(u,p) \in [0,M] \times P} |\partial_{uu} f(u;p)|,
\end{equation}
and the affine structure in Assumption~\ref{ass2} gives the Lipschitz constant of $\partial_u f$ in $p$,
\begin{equation}\label{eqKfTilde}
  \tilde K_f := Q\, L_\theta\, \max_{l \in \{1,\ldots,Q\}} \max_{u \in [0,M]} |f_l'(u)|.
\end{equation}
The pointwise bound $|\partial_u f(u_h(p);p) - \partial_u f(u_h(q);q)| \leq L_f'\, |w| + \tilde K_f\, \|p-q\|$, combined with $|\partial_u f| \leq L_f$, Cauchy--Schwarz, and Young's inequality, yields
\begin{equation}\label{eqBtildeBound}
\begin{aligned}
  |\widetilde{\mathcal{B}}(t)| \leq {}
  & \bigl(L_f + \tfrac{1}{2}\bigr)\, \|\partial_t w\|_{L^2}^2 \\
  &+ (L_f')^2\, \|\partial_t u_h(\cdot, t; p)\|_{L^\infty}^2\, \|w\|_{L^2}^2 \\
  &+ \tilde K_f^2\, \|\partial_t u_h(\cdot, t; p)\|_{L^2}^2\, \|p - q\|^2.
\end{aligned}
\end{equation}

Substituting~\eqref{eqAtildeBound} and~\eqref{eqBtildeBound} into~\eqref{eqDiffEqW}, folding $\tfrac{d_{\min}}{2}\|\partial_x \partial_t w\|_{L^2}^2$ into the diffusion term, multiplying by~$2$, and using $\|w\|_{L^2}^2 \leq G\, \|p - q\|^2$ from~\eqref{eq33}, we obtain
\begin{equation}\label{eqGronwallReadyW}
  \tfrac{d}{dt}\|\partial_t w\|_{L^2}^2 \leq (2 L_f + 1)\, \|\partial_t w\|_{L^2}^2 + \tilde\kappa(t)\, \|p - q\|^2,
\end{equation}
\begin{equation}\label{eqKappaTilde}
\begin{aligned}
  \tilde\kappa(t) := {}
  & \tfrac{L_d^2}{d_{\min}}\, \|\partial_x \partial_t u_h(\cdot, t; q)\|_{L^2}^2 + 2\, \tilde K_f^2\, \|\partial_t u_h(\cdot, t; p)\|_{L^2}^2\\
  &+ 2\, (L_f')^2\, G\, \|\partial_t u_h(\cdot, t; p)\|_{L^\infty}^2.
\end{aligned}
\end{equation}

The 1D Sobolev embedding $H^1(\Omega) \hookrightarrow L^\infty(\Omega)$~\cite[Th.~8.8, pp.~212]{brezis2011} gives the constant
\begin{equation}\label{eqCS}
  \|v\|_{L^\infty(\Omega)} \leq C_S\, \|v\|_{H^1(\Omega)} \qquad \forall v \in H^1(\Omega),
\end{equation}
so $\|\partial_t u_h\|_{L^\infty}^2 \leq C_S^2\, (\|\partial_t u_h\|_{L^2}^2 + \|\partial_x \partial_t u_h\|_{L^2}^2)$, and~\eqref{eqM2},~\eqref{eqE1} yield $\int_0^T \|\partial_t u_h\|_{L^\infty}^2\, \dd t \leq C_S^2\, (M_2 + E_1)$. Integrating~\eqref{eqKappaTilde} over $[0,T]$,
\begin{equation}\label{eqKappaIntegrated}
\begin{aligned}
  \int_0^T \tilde\kappa(t)\, \dd t \leq {}
  & \tfrac{L_d^2\, E_1}{d_{\min}} + 2\, \tilde K_f^2\, M_2\\
  &+ 2\, C_S^2\, (L_f')^2\, G\, (M_2 + E_1) .
\end{aligned}
\end{equation}

The initial value $\partial_t w(\cdot, 0;p;q)$ is controlled by evaluating~\eqref{eqDtUh0} at $p$ and at $q$ and using the $L^2$-stability of $P_{L^2}^{X_h}$, which yields $\|\partial_t w(\cdot, 0;p;q)\|_{L^2} \leq \tilde K_0\, \|p - q\|$, with
\begin{equation}\label{eqK0tilde}
  \tilde K_0 := L_d\, D_2 + d_{\max}\, \tilde D_2 + L_f\, K_0 + K_f\, \sqrt{L},
\end{equation}
\begin{equation}\label{eqD2tilde}
  \tilde D_2 := Q\, L_\xi\, \max_{l \in \{1,\ldots,Q\}} \|\partial_{xx} u_0^l\|_{L^2}.
\end{equation}
Gronwall's lemma applied to~\eqref{eqGronwallReadyW}, combined with~\eqref{eqKappaIntegrated}, yields
\begin{equation}\label{eq34}
  \sup_{t \in [0,T]} \|\partial_t w(\cdot, t)\|_{L^2}^2 \leq \tilde G\, \|p - q\|^2,
\end{equation}
where $\tilde G$ is the explicit constant
\begin{multline}\label{eqGtilde}
  \tilde G := e^{(2 L_f + 1) T}\Big[\tilde K_0^2 + \tfrac{L_d^2 E_1}{d_{\min}}\\
  + 2 C_S^2 (L_f')^2 G (M_2 + E_1)
  + 2\tilde K_f^2 M_2\Big].
\end{multline}

Now, to show Lipschitzness of $p \mapsto E(p)$, we let
\begin{equation}\label{lambda}
  \lambda := \sup_{v \in X_h \setminus \{0\}} \frac{\|R_r v\|_{L^2}}{\|v\|_{L^2}}.
\end{equation}
The projector $R_r: X_h \rightarrow X_r \subset L^2(\Omega)$ is a linear map between finite-dimensional subspaces of $L^2$; hence, bounded in the $L^2$ norm.

Now, for $\psi_p := u_h(\cdot, \cdot; p) - R_r u_h(\cdot, \cdot; p)$, having $\psi_p - \psi_q = (I - R_r) w$ and $\partial_t \psi_p - \partial_t \psi_q = (I - R_r) \partial_t w$, we use~\eqref{lambda} to obtain
\begin{equation}\label{eqPsiDiff}
  \|\psi_p - \psi_q\|_{L^2} \leq (1 + \lambda)\, \|w\|_{L^2},
\end{equation}
\begin{equation}\label{eqDtPsiDiff}
  \|\partial_t \psi_p - \partial_t \psi_q\|_{L^2} \leq (1 + \lambda)\, \|\partial_t w\|_{L^2}.
\end{equation}
Combining~\eqref{eqPsiDiff},~\eqref{eqDtPsiDiff} with~\eqref{eq33},~\eqref{eq34}, the identity $|a^2 - b^2| = (a + b)|a - b|$, Cauchy--Schwarz in time, and $\|u_h(\cdot, \cdot; p)\|_{L^2} \leq M\, \sqrt{L}$, we obtain
\begin{equation}\label{eq36}
  |E(p) - E(q)| \leq C_2\, \|p - q\|,
\end{equation}
\begin{equation}\label{eqC2}
  C_2 := 2\, (1 + \lambda)^2\, \Bigl[M\, \sqrt{L\, G} + \sqrt{M_2}\, \sqrt{T\, \tilde G}\Bigr].
\end{equation}

Next, we estimate $E(p_j)$ at each snapshot parameter, using the optimality of the POD basis, giving~\cite[Prop.~1, Lem.~3]{kunisch2001}
\begin{multline}\label{eqKV}
  \sum_{k=1}^{n_t} \|u_h(\cdot, t_k; p_j) - P_r u_h(\cdot, t_k; p_j)\|_{L^2}^2 \\
  + \sum_{k=1}^{n_t-1} \|\overline\partial u_h(\cdot, t_k; p_j) - P_r \overline\partial u_h(\cdot, t_k; p_j)\|_{L^2}^2 \\
  \leq \sum_{i=r+1}^{n_s} \sigma_i^2,
\end{multline}
where $P_r : X_h \to X_r$ is the $L^2$-orthogonal projector, verifying
\begin{equation}\label{eq35}
  \|v - R_r v\|_{L^2} \leq (1 + \lambda)\, \|v - P_r v\|_{L^2} \qquad \forall v \in X_h.
\end{equation}
Indeed, $P_r v \in X_r$ and the Ritz projector $R_r$ fixes every element of $X_r$. Hence $R_r P_r v = P_r v$, which gives $v - R_r v = (I - R_r)(v - P_r v)$. Taking the $L^2$ norm on both sides and using $\|I - R_r\| \leq 1 + \lambda$, \eqref{eq35} follows.

To bound $E(p_j)$, we need to bound $\sup_{t \in [0,T]} \|\psi_{p_j}(\cdot, t)\|_{L^2}^2$ and $\int_0^T \|\partial_t \psi_{p_j}(\cdot, t)\|_{L^2}^2\, \dd t$, with $\psi_{p_j}(\cdot, t) := u_h(\cdot, t; p_j) - R_r u_h(\cdot, t; p_j)$. Since $\psi_{p_j}(\cdot, t) = \psi_{p_j}(\cdot, t_k) + \int_{t_k}^{t} \partial_t \psi_{p_j}(\cdot, s)\, \dd s$, the triangular and Cauchy--Schwarz inequalities give
\begin{multline}\label{eqSupSub}
  \sup_{t \in [t_k, t_{k+1}]} \|\psi_{p_j}(\cdot, t)\|_{L^2}^2 \\
  \leq 2\, \|\psi_{p_j}(\cdot, t_k)\|_{L^2}^2 + 2\, \tau\, \int_{t_k}^{t_{k+1}} \|\partial_t \psi_{p_j}(\cdot, s)\|_{L^2}^2\, \dd s.
\end{multline}

The quotient $\overline\partial \psi_{p_j}(\cdot, t_k) = \tau^{-1} \int_{t_k}^{t_{k+1}} \partial_t \psi_{p_j}(\cdot, \sigma)\, \dd \sigma$ is the average of $\partial_t \psi_{p_j}$ over the subinterval. Its deviation from the pointwise value satisfies $\partial_t \psi_{p_j}(\cdot, s) - \partial_t \psi_{p_j}(\cdot, \sigma) = \int_\sigma^s \partial_{tt} \psi_{p_j}(\cdot, r)\, \dd r$. Applying Jensen's inequality, we obtain
\begin{multline}\label{eqIntSub}
  \int_{t_k}^{t_{k+1}} \|\partial_t \psi_{p_j}(\cdot, s) - \overline\partial \psi_{p_j}(\cdot, t_k)\|_{L^2}^2\, \dd s \\
  \leq \tau^2\, \int_{t_k}^{t_{k+1}} \|\partial_{tt} \psi_{p_j}(\cdot, s)\|_{L^2}^2\, \dd s.
\end{multline}
We now sum~\eqref{eqSupSub} and~\eqref{eqIntSub} over $k$ and using~\eqref{eq35} and~\eqref{eqKV}, we obtain
\begin{multline}\label{eqEpjResidual}
  E(p_j) \leq (1 + \lambda)^2\, (1 + T)\, \sum_{i=r+1}^{n_s} \sigma_i^2 \\
  + \tau^2\, \int_0^T \|\partial_{tt} \psi_{p_j}(\cdot, t)\|_{L^2}^2\, \dd t.
\end{multline}

Since $\psi_{p_j} = u_h(\cdot, \cdot; p_j) - R_r u_h(\cdot, \cdot; p_j)$ and $R_r$ is linear, $\|\partial_{tt} \psi_{p_j}\|_{L^2} \leq (1 + \lambda)\, \|\partial_{tt} u_h(\cdot, \cdot; p_j)\|_{L^2}$. Using~\eqref{eqTheta}, we obtain
\begin{equation}\label{eqThetaPsi}
  \sup_{p \in P} \int_0^T \|\partial_{tt} \psi_p(\cdot, t)\|_{L^2}^2\, \dd t \leq (1 + \lambda)^2 \Theta< \infty.
\end{equation}
Hence, by taking $\tau > 0$ such that
\begin{equation}\label{eqSnapshotResolution}
  \tau^2\, (1 + \lambda)^2\, \Theta \leq (1 + T)\, \sum_{i=r+1}^{n_s} \sigma_i^2,
\end{equation}
we obtain
\begin{equation}\label{eqEpj}
  E(p_j) \leq C_1\, \sum_{i=r+1}^{n_s} \sigma_i^2,
\end{equation}
\begin{equation}\label{eqC1def}
  C_1 := 4\, (1 + T)^2\, (1 + \lambda)^2.
\end{equation}
Combining~\eqref{eq36} and~\eqref{eqEpj}, using the radius $\rho_P$, \eqref{eq25} follows.

\section{Numerical Illustrations}\label{sec4}
We test our approach on instances of $\Sigma$ with
$\Omega=(0,1)$ and $T=1$. We consider  $\nh=100$ nodes and time step
$\tau=0.01$ for the FEM. The first example has uncertainty only in the dynamics; the second has uncertainty in both the dynamics and the
initial condition. 
The two models are:
\\
\smallskip
\noindent\textbf{Allen--Cahn}:
In this case, $f(u;p_1) := u(1\!-\!u)(u\!-\!p_1)$, $p_1\!\in\![0.3,0.7]$,
$d = p_2 \in [0.08,0.12]$,
$u_0(x) := 0.5+0.1\cos(\pi x)$. 

\noindent\textbf{Logistic}:
In this case, $f(u;p_1)=p_1\,u(1\!-\!u)$, $d=0.01$,
$p_1\!\in\![0.8,1.2]$,
$u_0(x;p_2)=p_2 \sin^2(\pi x/2)$, $p_2\!\in\![0.5,1.5]$.
Here $p_1$ enters the dynamics and $p_2$ the initial condition.

%\subsection{Assumptions, POD basis, and error bound}
Both models satisfy Assumptions~\ref{ass1}
and~\ref{ass2}: Allen--Cahn with $\dmin=0.08$, $\dmax=0.12$,
$M=1$, $\|u_0\|_{L^\infty}\leq 0.6$, and Logistic with $d=0.01$,
$M=1.5$, $\|u_0\|_{L^\infty}\leq 1.5$.
The reactions are polynomial in $u$ and the
initial conditions are $C^\infty(\overline\Omega)$ with
$\partial_x u_0^l|_{\partial\Omega}=0$, so the regularity
required by Proposition~\ref{prop3} is met.
For Allen--Cahn, $r=2$ suffices ($S\in\R^{100\times7960}$, relative tail energy $\sum_{i=r+1}^{n_s}\sigma_i^2 / \sum_{i=1}^{n_s}\sigma_i^2\leq 5\times10^{-7}$, $\rho_P\approx 2.9\times 10^{-2}$). For Logistic, the independent action of $p_1,p_2$ requires $r=6$ ($S\in\R^{100\times9900}$, tail energy $\leq 9\times10^{-6}$, $\rho_P\approx 7.0\times 10^{-2}$). 

%\subsection{Over-approximation of the ROM reachable set}
We compute $\Ccal(t)$ using CORA~\cite{althoff2010}, which propagates zonotopes through nonlinear ODEs with polynomial enclosures and satisfies $d_H(\Rr(t),\Ccal(t))\leq\eta$, so both conclusions of Theorem~\ref{thm1} hold. Parameters in $g$~\eqref{eq14} are tracked as extra states with zero derivative; those entering only the initial condition are absorbed into $\Ccal(0)$. Since $\eta=O(\mathrm{diam}(P)^2)$, we split $P$ into $K^2=16$ sub-boxes ($K=4$), run CORA on each (\texttt{alg=poly}, \texttt{tensorOrder=3}, $\tau=0.01$), and take the union.
%\subsection{Results}
Figures~\ref{fig:ex1_band}--\ref{fig:ex2_band}
show $\Rhat(t)$ enclosing the 200 test
trajectories at three time instances.
Table~\ref{tab:both} reports the error
components at $T=1$.
\begin{figure}[h]
\centering
\includegraphics[width=0.84\columnwidth]{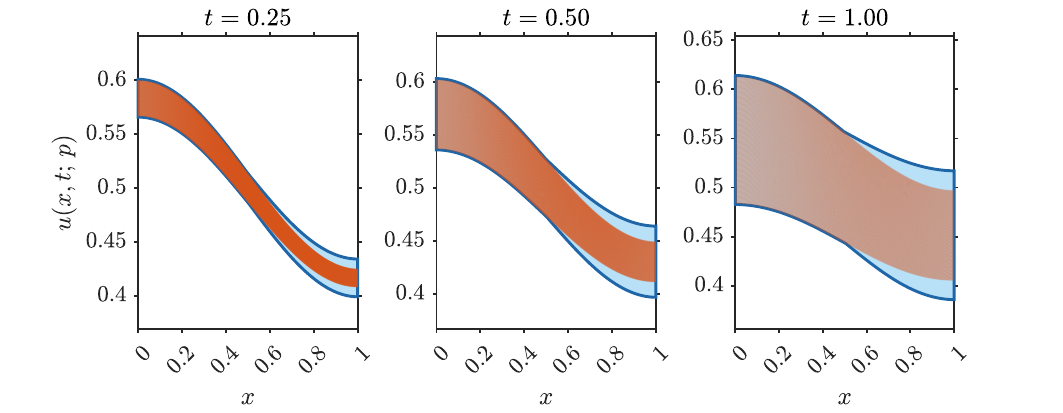}
\caption{Allen--Cahn: $\Rhat(t)$ (blue) and
200 FEM trajectories (orange).}
\label{fig:ex1_band}
\end{figure}
\begin{figure}[h]
\centering
\includegraphics[width=0.84\columnwidth]{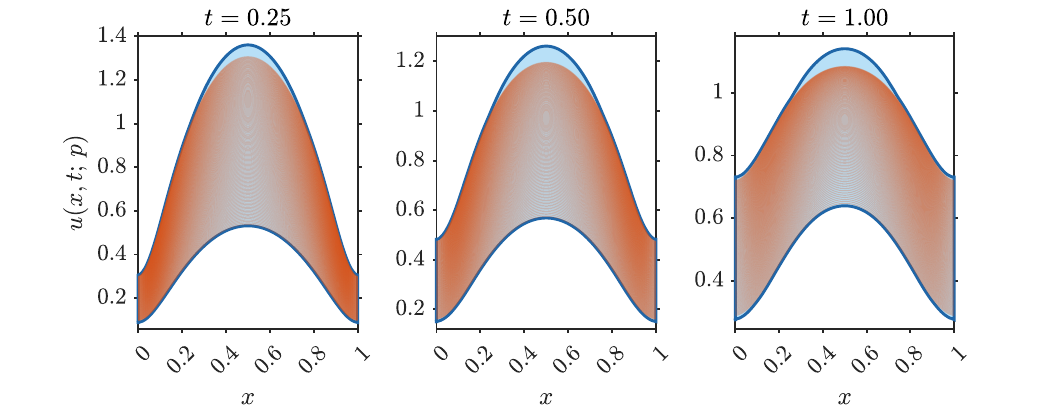}
\caption{Logistic PDE: $\Rhat(t)$ (blue) and
200 FEM trajectories (orange).}
\label{fig:ex2_band}
\end{figure}
\begin{table}[h]
\caption{}
\label{tab:both}
\centering
\setlength{\tabcolsep}{3pt}
\small
\begin{tabular}{lccccc}
\hline
 & $r$ & $\varepsilon_h$ & $\varepsilon_r$ &
   $\eta$ &  $t_{\rm on}$ \\
\hline
A.--C.
  & 2 & 1.7e-4 & 7.3e-4 & 3.1e-4
   & 186\,s \\[2pt]
Log.
  & 6 & 2.8e-4 & 2.9e-3 & 4.9e-3
  & 700\,s \\
\hline
\end{tabular}
\end{table}
For Allen--Cahn, the three errors $(\varepsilon_r,\varepsilon_h,\eta)$ in are comparable (all near $10^{-4}$); see Table~\ref{tab:both}, leading to the mismatch $2\varepsilon_h + 2\varepsilon_r + \eta = 2.1\times10^{-3}$. 

For the Logistic model, we note that $\eta > \varepsilon_r >> \varepsilon_h$. Tightening the bound in this case suggests more POD modes or a finer split of $P$. However, increasing $r$ slows each CORA run ($186$\,s at $r=2$ vs.
\ $700$\,s at $r=6$ for $K^2=16$), and increasing $K$ reduces $\eta$ quadratically but multiplies the number of runs by $K^2$. 

\vspace*{-5pt}
\section{Conclusion} \label{sec5}
We presented a method to approximate reachable sets for uncertain nonlinear reaction-diffusion equations. Proposition~\ref{prop3} establishes the bound on the model-reduction error, uniformly over the parameter set, under structural conditions on the parametric dependence, and Theorem~\ref{thm1} uses such bounds, together with the FEM and existing set-based reachability results for ODEs, to guarantee that the over-approximation set is sound. In future work, we envision relaxing these conditions to cover broader classes of nonlinear reaction-diffusion systems, including coupled systems and higher-dimensional spatial domains.


\vspace*{-5pt}
\begin{thebibliography}{99}
\bibitem{murray2003} J.~D.~Murray, \emph{Mathematical Biology II}, 3rd~ed., Springer, 2003.
\bibitem{pao1992} C.~V.~Pao, \emph{Nonlinear Parabolic and Elliptic Equations}, Plenum Press, 1992.
\bibitem{ouchdiri2025optimal} M.~A.~Ouchdiri, H.~Faquir, S.~Benjelloun, M.~Maghenem, I.~Otero-Muras, and A.~Saoud, ``An optimal-control framework for reaction diffusion systems with application to synthetic developmental biology,'' \emph{IEEE Conf. Decision and Control}, pp.~1925--1930, 2025.
\bibitem{girard2005} A.~Girard, ``Reachability of uncertain linear systems using zonotopes,'' \emph{Proc. HSCC}, LNCS 3414, pp.~291--305, 2005.
\bibitem{scott2013} J.~K.~Scott and P.~I.~Barton, ``Improved relaxations for the parametric solutions of ODEs using differential inequalities,'' \emph{J. Global Optim.}, vol.~57, pp.~143--176, 2013.
\bibitem{althoff2021} M.~Althoff, G.~Frehse, and A.~Girard, ``Set propagation techniques for reachability analysis,'' \emph{Annu. Rev. Control Robot. Auton. Syst.}, vol.~4, pp.~369--395, 2021.
\bibitem{kharkovskaia2016} T.~Kharkovskaia, D.~Efimov, A.~Polyakov, and J.~P.~Richard, ``Interval observers for PDEs: approximation approach,'' \emph{IFAC-PapersOnLine}, vol.~49, no.~18, pp.~915--920, 2016.
\bibitem{tran2018reachPDE} H.-D.~Tran, W.~Xiang, S.~Bak, and T.~T.~Johnson, ``Reachability analysis for one dimensional linear parabolic equations,'' \emph{IFAC-PapersOnLine}, vol.~51, no.~16, pp.~7--12, 2018.
\bibitem{kunisch2001} K.~Kunisch and S.~Volkwein, ``Galerkin proper orthogonal decomposition methods for parabolic problems,'' \emph{Numer. Math.}, vol.~90, pp.~117--148, 2001.
\bibitem{thomee2006} V.~Thom\'ee, \emph{Galerkin Finite Element Methods for Parabolic Problems}, 2nd~ed., Springer, 2006.
\bibitem{meyer2018} P.-J.~Meyer, A.~Devonport, and M.~Arcak, \emph{Interval Reachability Analysis}, Springer, 2021.
\bibitem{althoff2010} M.~Althoff, O.~Stursberg, and M.~Buss, ``Computing reachable sets of hybrid systems using a combination of zonotopes and polytopes,'' \emph{Nonlinear Anal. Hybrid Syst.}, vol.~4, pp.~233--249, 2010.
\bibitem{le2010reachability} C.~Le~Guernic and A.~Girard, ``Reachability analysis of linear systems using support functions,'' \emph{Nonlinear Anal. Hybrid Syst.}, vol.~4, no.~2, pp.~250--262, 2010.
\bibitem{rungger2018accurate} M.~Rungger and M.~Zamani, ``Accurate reachability analysis of uncertain nonlinear systems,'' \emph{Proc. HSCC}, pp.~61--70, 2018.


\bibitem{brezis2011}
H.~Brezis,
\emph{Functional Analysis, Sobolev Spaces and Partial Differential Equations}.
New York: Springer, 2011.

\bibitem{hartman1982}
P.~Hartman,
\emph{Ordinary Differential Equations}, 2nd~ed.,
Classics in Applied Mathematics, vol.~38.
Philadelphia, PA: SIAM, 2002.

\bibitem{crouzeix1987}
M.~Crouzeix and V.~Thom\'ee,
``The stability in $L_p$ and $W_p^1$ of the $L_2$-projection onto finite element function spaces,''
\emph{Math. Comp.}, vol.~48, no.~178, pp.~521--532, 1987.
\end{thebibliography}
\end{document}